\documentclass[10pt, a4paper,
oneside, headinclude,footinclude]{scrartcl}
\usepackage[utf8]{inputenc}

\usepackage[
nochapters, 
pdfspacing, 
dottedtoc 
]{classicthesis} 
\usepackage{amsopn}
\usepackage[T1]{fontenc} 
\usepackage{multirow}
\usepackage[utf8]{inputenc} 
\usepackage{hhline}
\usepackage{graphicx} 
\graphicspath{{Figures/}} 
\usepackage{indentfirst}
\usepackage{enumitem} 
\usepackage{savesym}
\usepackage{mathrsfs}
\usepackage{geometry}
\usepackage{esint}
\usepackage{longtable}
\usepackage{float}
\usepackage{adjustbox}

\usepackage[
    backend=biber,
    style=numeric,
    natbib=false,
    url=true, 
    doi=false,
    eprint=true,
    maxnames=50
]{biblatex}

\usepackage{stmaryrd}

\usepackage{subfig} 

\usepackage{amsmath, amssymb,amsthm,amsfonts} 

\usepackage[english,capitalize]{cleveref}

\usepackage{thmtools}
\usepackage{thm-restate}

\theoremstyle{plain}
\newtheorem{teorema}{Theorem}[section]
\newtheorem{proposizione}[teorema]{Proposition}
\newtheorem{lemma}[teorema]{Lemma}
\newtheorem{corollario}[teorema]{Corollary}
\newtheorem*{theorem*}{Theorem}

\theoremstyle{definition}
\newtheorem{definizione}{Definition}[section]

\theoremstyle{remark}
\newtheorem{osservazione}{Remark}[section]

\newcommand{\Mass}{\mathbb{M}}
\newcommand{\R}{\mathbb{R}}
\newcommand{\N}{\mathbb{N}}

\newcommand{\Leb}{\mathscr{L}}

\newcommand{\Lip}{\mathrm{Lip}}

\newcommand{\supp}{\mathrm{supp}}

\newcommand{\Flat}{\mathbb{F}}
\newcommand{\dV}{d_V\kern-1pt}
\newcommand{\dW}{d_W\kern-1pt}

\newcommand{\trait}[3]{\vrule width #1ex height #2ex depth #3ex}
\newcommand{\trace}{\mathchoice%
  {\mathbin{\trait{.12}{1.2}{.03}\trait{.8}{0.09}{0.03}}}
  {\mathbin{\trait{.12}{1.2}{.03}\trait{.8}{0.09}{0.03}}}
  {\mathbin{\hskip.15ex\trait{.09}{.84}{0.02}\trait{.56}{.07}{.02}}\hskip.15ex}
  {\mathbin{\trait{.07}{.6}{.01}\trait{.4}{.06}{.01}}}}

\newcounter{const}

\newcounter{eps}

\newcommand{\vertiii}[1]{{\left\vert\kern-0.25ex\left\vert\kern-0.25ex\left\vert #1 
    \right\vert\kern-0.25ex\right\vert\kern-0.25ex\right\vert}}

\hypersetup{
colorlinks=true, breaklinks=true,bookmarksnumbered,
urlcolor=webbrown, linkcolor=RoyalBlue, citecolor=webgreen, 
pdftitle={}, 
pdfauthor={\textcopyright}, 
pdfsubject={}, 
pdfkeywords={}, 
pdfcreator={pdfLaTeX}, 
pdfproducer={LaTeX with hyperref and ClassicThesis} 
} 

\title{\normalfont\spacedlowsmallcaps{A simple proof of the $1$-dimensional\\ flat chain conjecture}} 
\author{\spacedlowsmallcaps{Andrea Marchese\textsuperscript{*}, Andrea Merlo\textsuperscript{**}}}

\date{}

\begin{document}

\renewcommand{\sectionmark}[1]{\markright{\spacedlowsmallcaps{#1}}} 
\lehead{\mbox{\llap{\small\thepage\kern1em\color{halfgray} \vline}\color{halfgray}\hspace{0.5em}\rightmark\hfil}} 
\pagestyle{scrheadings}
\maketitle 
\setcounter{tocdepth}{2}

{\let\thefootnote\relax\footnotetext{* \textit{Università di Trento, Via Sommarive 14, 38123 Trento,  Italy.}}}
{\let\thefootnote\relax\footnotetext{** \textit{Universidad del Pa\'is Vasco (UPV/EHU), Barrio Sarriena S/N 48940 Leioa, Spain.}}}

{\rightskip 1 cm
\leftskip 1 cm
\parindent 0 pt
\footnotesize

	%
{\textsc Abstract.}
We give a new, elementary proof of the fact that metric 1-currents in the Euclidean space correspond to Federer-Fleming flat chains.
 

\par
}


%
%

\section*{Acknoweldgments}
A. Ma. is partially supported by the PRIN project 2022PJ9EFL "Geometric Measure Theory: Structure of Singular Measures, Regularity Theory and Applications in the Calculus of Variations" and by GNAMPA-INdAM.
A. Me. is supported by the European Union’s Horizon Europe research and innovation programme under the Marie Sk\l odowska-Curie grant agreement no 101065346.

\section{Introduction}

In this note we give a new, elementary proof of the following result, see \S \ref{s2} for the relevant notation.

\begin{teorema} Let $T$ be a metric $1$-current in $\R^d$ and denote by $\tilde{T}$ the classical current induced by $T$. Then, $\tilde{T}$ is a flat chain.     
\end{teorema}

The fact that Ambrosio-Kirchheim metric $1$-currents correspond to Federer-Fleming flat chains was proved by Schioppa in \cite{Schioppa}, by D. Bate, the first named author and G. Alberti in \cite{AlbMar2}, and by De Masi and the first named author in \cite{DeMaMa}. The analogue result for metric $d$-currents was proved in \cite{DPR}. All these proofs use, directly or indirectly, the notion of \emph{Alberti representations} of a measure $\mu$ and, correspondingly, the construction of \emph{width functions}, that is, Lipschitz functions of small supremum norm with high derivative $\mu$-a.e. along certain directions. Our proof, instead, is immediate, it does not rely on any such notion or construction, and it sheds new light on the PDE challenges behind the possibility to prove the conjecture in full generality.

\section{Notation and preliminary results}
\label{s2}
\subsection{Classical currents in $\R^d$}
A \emph{$k$-dimensional current} $T$ in $\R^d$ is a continuous linear functional on the 
space of smooth and 
compactly supported differential $k$-forms on $\R^d$, endowed with the topology of test functions. 
The \emph{boundary} of $T$,
$\partial T$, is the $(k-1)$-current defined  via $\langle\partial T, \omega \rangle := \langle T, d\omega\rangle $
for every smooth and compactly supported $(k-1)$-form $\omega$.
The \emph{mass} of $T$, denoted by 
$\Mass(T)$, is the supremum of $\langle T, \omega\rangle$ over
all $k$-forms $\omega$ such that $\|\omega\|\le 1$, where $\|\omega\|$ denotes the comass norm.  A current $T$ is called \emph{normal} if both $T$ 
and $\partial T$ have finite mass. 

By the Radon--Nikod\'{y}m theorem, a $k$-dimensional current $T$ with finite mass can be written in an essentially unique way in the form
$T=\tau_T\mu_T$ where $\mu_T$ is a finite positive measure
and $\tau_T$ is a $k$-vector field with unit mass norm $\mu_T$-a.e.
In particular, the action of $T$ on a smooth and compactly supported $k$-form
$\omega$ is given by
\[
\langle T, \omega\rangle 
= \int_{\R^d} \langle\omega(x), \tau_T(x)\rangle d\mu_T(x)
\; .
\]

On the space of smooth and compactly supported differential $k$-forms, we consider the flat norm, see \cite[\S 4.1.12]{Federer1996GeometricTheory}, 
    $$\mathbb{F}(\phi):=\max\{\|\omega\|,\|d\omega\|\}.$$ 
This induces a corresponding flat seminorm on currents defined by
$$\mathbb{F}(T)=\sup\{T(\omega):\omega\in X, \mathbb{F}(\omega)\leq 1\}.$$
We recall that, by \cite[\S 4.1.12]{Federer1996GeometricTheory}, we have 
$$\mathbb{F}(T)=\min\{\mathbb{M}(R)+\mathbb{M}(S):T=R+\partial S\}.$$
More information on currents can be found in \cite{Federer1996GeometricTheory}.

\subsection{Metric currents}\label{ss:metriccurr}

Let $(X,d)$ be a complete metric space. We denote $\mathscr{D}^k(X):=\Lip_b(X,\R)\times\Lip(X,\R)^k$, where $\Lip(X,\R)$ is the space of Lipschitz functions on $X$ and $\Lip_b(X,\R)$ is the subspace of bounded Lipschitz functions . 
\begin{definizione}[Metric currents]\label{d:metric}
A multilinear functional $T:\mathscr{D}^k(X)\to\R$ is said to be a $k$-dimensional \emph{metric current} if 
\begin{itemize}
    \item [(i)] \emph{continuity}: for every $f\in \Lip_b(X,\R)$, $(\pi_1^n)_{n\in\mathbb{N}},\dots,(\pi_k^n)_{n\in\mathbb{N}}\subset \Lip(X,\R)$ converging pointwise to $\pi_1,\dots,\pi_k$ with $\Lip(\pi_i^n)\leq C$ for every $n$
    $$T(f,\pi_1^n,\dots,\pi_k^n)\to T(f,\pi_1,\dots,\pi_k);$$
    \item [(ii)] \emph{locality}: if there exists $i\in\{1,\dots,k\}$ such that $\pi_i\equiv c$ on a neighbourhood of $\supp f$ then $T(f,\pi_1,\dots,\pi_k)=0$;
    \item [(iii)] \emph{finite mass}: there exists a finite Radon measure $\mu$ such that
\begin{equation}\label{e:mass}
|T(f,\pi_1,\dots,\pi_k)|\leq \Lip(\pi_1)\cdots\Lip(\pi_k)\int_X|f|d\mu.
\end{equation}
The minimal measure with such property is denoted by $\mu_T$.

\end{itemize}
More information on metric currents can be found in \cite{metriccurrents} and \cite{lang}. 
\end{definizione}

Let's shift our attention to the case $X = \R^d$, equipped with the Euclidean distance. It's worth recalling that for every $k$-dimensional metric current $T$ on $\R^d$ with compact support, there exists a corresponding "classical" $k$-dimensional current $\tilde T$, see \cite[Theorem 11.1]{metriccurrents}. Denoting $\Lambda(k,d)$ the set of multi-indices $\alpha=(1\leq\alpha_1<\dots<\alpha_k\leq d)$ of length $k$ in $\R^d$, the condition defining $\tilde T$ is that for every smooth and compactly supported $k$-form $$\omega=\sum_{\alpha\in\Lambda(k,d)}\omega_\alpha dx_{\alpha_1}\wedge\dots\wedge dx_{\alpha_k}$$ 
it holds
\begin{equation}\label{e:deftilda}
    \langle \tilde T,\omega\rangle=\sum_{\alpha\in\Lambda(k,d)}T(\omega_\alpha, x_{\alpha_1},\dots, x_{\alpha_k}).
\end{equation}

Conversely, for every flat chain $T$ with finite mass and compact support, there exists a corresponding metric current $\hat T$. These mappings are inverses of each other when restricted to normal currents, see \cite[Theorem 5.5]{lang}.

\section{Properties of purely non-flat currents}

In \cite{AlbMar2}, it has been noticed that normal currents are closely related to flat chains with finite mass, proving that, in codimension at least 1, every flat chain with finite mass is the restriction of a normal current to a Borel set. For this reason many geometric properties valid for normal currents can be inferred also for flat chains with finite mass.

This note originates from the observation of some interesting properties enjoyed by those currents which do not contain any portion of a flat chain with finite mass, for which we give the following definition.

\begin{definizione}[Purely non-flat current]
We say that a current $T$ with finite mass is \emph{purely non-flat} if for every decomposition $T=T_1+T_2$ with $\Mass(T)=\Mass(T_1)+\Mass(T_2)$ where $T_1$ is a flat chain, than $T_1=0$.
\end{definizione}

\begin{osservazione}
    Let us observe that a $k$-current of finite mass $T=\tau_T\mu_T$ is purely non-flat if and only if a certain pointwise relation between the measure and the $k$-vector field holds, that is, if and only if $\tau_T(x)\not\in V_k(\mu_T,x)$ for $\mu$-almost every $x$, see \cite[Definition 4.1, Theorem 1.2]{AlbMar}.
\end{osservazione}


We define the \emph{closed flat norm} of a current $T$ as
$$\Flat_0(T):=\sup\{\langle T,\omega\rangle:\|\omega\|\leq 1, d\omega=0\}.$$
In general this is not equivalent to the flat norm, indeed this quantity equals zero whenever $\partial T=0$. However, we will show that the two quantities are equal for purely non-flat currents with finite mass.

\begin{proposizione}\label{p:flatmass}
Assume that $T$ is a purely non-flat current and $\Mass(T)<\infty$. Then $\Flat(T)=\Mass(T)$. 
\end{proposizione}

\begin{proof}
     Since $\Mass(T)$ is finite then $\Flat(T)<\infty$. To prove that $\Flat(T)=\Mass(T)$, write $T=R+\partial S$ with $\Mass(R)+\Mass(S)=\Flat(T)$, see \cite[\S 4.1.12]{Federer1996GeometricTheory}. Since $S$ is normal, the fact that $T$ is purely non-flat implies that $\partial S=-R\trace A$ for some Borel set $A$. The minimality of $\Mass(R)+\Mass(S)$ then implies $S=0$, hence $R=T$, so that $\Flat(T)=\Mass(R)=\Mass(T)$.
\end{proof}

\begin{proposizione}\label{p:newflat}
Assume that $T$ is a purely non-flat current and $\Mass(T)<\infty$. Then $\Flat(T)=\Flat_0(T)$.    
\end{proposizione}

\begin{proof}
    The fact that $\Flat_0\leq \Flat$ is true for every current. Towards a proof by contradiction that $\Flat_0(T)=\Flat(T)$ for a purely non-flat current $T$ with finite mass, assume that $$f_0:=\Flat_0(T)<\Flat(T)=:f_1.$$
    This implies that $T$, seen as a linear functional on the space $Y$ of smooth and compactly supported forms $\omega$ with $d\omega=0$, endowed with the flat norm (which in $Y$ coincides with the comass norm), has operator norm equal to $f_0$.

    By Hahn-Banach theorem, $T$ can be extended to a linear functional $W$ on the space $X$ of smooth and compactly supported forms, endowed with the flat norm, which coincides with $T$ on $Y$. This implies that $\partial T=\partial W$ and moreover $\Flat(W)=f_0$. We can write $W=R+\partial S$, with $\Mass(R)+\Mass(S)=f_0$. In particular $\partial R- \partial T=0$, that is $T-R=\partial N$ for some normal current $N$. We write $R=R_p+R_f$ with $R_p$ purely non-flat and $R_f$ a flat chain with $\Mass(R_p)+\Mass(R_f)=\Mass(R)$. This can be done by maximizing the mass of the restriction of $R$ to a Borel set $A$ such that $R\trace \mathbb{1}_A$ is a flat chain. We deduce that $T-R_p=\partial N+R_f$. Being the right hand-side a flat chain while and the left hand side purely non-flat, the only possibility is that $T-R_f=0$, which is impossible because, by \cref{p:flatmass}, $\Mass(T)=f_1>f_0\geq\Mass(R)$.
\end{proof}

\section{Equivalence between metric 1-currents and flat chains in the Euclidean space}

We note that \cref{p:flatmass} and  \cref{p:newflat} yield an elementary proof of the fact that 1-dimensional metric currents in the Euclidean space correspond to Federer-Fleming flat chain. We begin with the following Lemma.


\begin{lemma}\label{l:reform}
For every smooth $1$-form $\omega$ such that $d\omega=0$ and $\lVert \omega\rVert\leq 1$ there exists $\pi\in C^\infty(\R^d)$ such that 
$$d\pi=\omega\qquad\text{and}\qquad \mathrm{Lip}(\pi)\leq 1.$$
\end{lemma}

\begin{proof}
    This is an immediate consequence of the fact that the exterior derivative coincides with the differential for smooth functions.
\end{proof}


\begin{proposizione}\label{prop:singcurves}
Let $\eta$ be a positive and finite Radon measure on $\R^d$ and let $\mu$ be a Radon measure singular with respect to the Lebesgue measure $\Leb^d$. Then, there exists a set of full measure of vectors $v\in \R^d$ such that $\eta$ and $({\tau_{v}})_\sharp\mu$ are mutually singular, where $\tau_v$ denotes the map $\tau_v(x):=v+x$.
\end{proposizione}

\begin{proof}
    Without loss of generality we can assume that $\mu$ is supported in $B(0,1)$.
   Let $A\subset\R^d$ be a Borel set such that $\Leb^d(A)=0$ and $\eta(A^c)=0$ and observe that by Tonelli's theorem we have
\begin{equation}
    \begin{split}
\int_{B(0,1)}({\tau_{v}})_\sharp\mu(A)d\Leb^d(v)=&\int_{B(0,1)}\int \mathbb{1}_{A}(z+v)d\mu(z)d\Leb^d(v)\\
=&\int\int_{B(0,1)} \mathbb{1}_{A}(z+v)d\Leb^d(v)d\mu(z)=\int \Leb^d(A-z)d\mu(z)=0.
    \end{split}
\end{equation}
    The above computation implies that $({\tau_{v}})_\sharp\mu(A)=0$ for $\Leb^d$-almost every $v\in B(0,1)$, so that for those $v$'s the measures $\eta$ and $({\tau_{v}})_\sharp\mu$ are mutually singular. 
\end{proof}

\begin{proposizione}\label{propfudb}
    Let $T$ be a metric $1$-current such that $\tilde T$ is purely non-flat. Then 
    $$\limsup_{|v|\to 0}\mathbb{F}(\tilde T-({\tau_{v}})_\sharp\tilde T)=0.$$
\end{proposizione}

\begin{proof}
By Proposition \ref{p:newflat}, in this case $\mathbb{F}(\tilde T)=\Flat_0(\tilde T)$. Towards a proof by contradiction of the proposition, assume that there exists a purely non-flat metric current $T$ for which 
$$\limsup_{\lvert v\rvert\to 0}\mathbb{F}(\tilde T-({\tau_{v}})_\sharp\tilde T)>c>0.$$
This implies that there exists a sequence of 
smooth, closed $1$-forms $(\omega_n)$ with 
$\Flat(\omega_n)\leq 1$ for every $n=1,2,\dots$ and a sequence of vectors $v_n\to 0$ such that 
$$\langle\tilde T-({\tau_{v_n}})_\sharp\tilde T,\omega_n\rangle>c(1-n^{-1})\quad \mbox{ for every $n$.}$$

By \cref{l:reform}, for every $n=1,2,\dots$ we can find $\pi_n\in C^\infty(\R^d)$ such that $\Lip(\pi_n)\leq 1$ and $\omega_n=d\pi_n$.
Possibly subtracting a constant, we can assume that $\pi_n(0)=0$, for every $n$, so that we can find a 1-Lipschtz function $\pi_\infty$ such that, up to non-relabeled subsequences,
    \begin{equation}
 \pi_{n}\to \pi_{\infty}\quad\text{locally uniformly.}
        \label{convdebole*}
    \end{equation}
   We deduce that for every $n$
   \begin{equation*}
          \begin{split}
      c(1-n^{-1})<& \langle\tilde T-{(\tau_{v_n})}_\sharp\tilde T,\omega_n\rangle=\langle\tilde T,\omega_n\rangle-\langle{(\tau_{v_n})}_\sharp\tilde T,\omega_n\rangle\\
     =&\langle\tilde T,\omega_n\rangle-\langle{\tilde T,(\tau_{v_n})}^\sharp\omega_n\rangle=T(1,\pi_n)-T(1,\pi_n\circ\tau_{v_n}).
       \end{split}    
   \end{equation*}
Thanks to the continuity of metric currents, the equi-Lipschitzianity of the $\pi_n$'s, and \eqref{convdebole*}, we reach a contradiction.
\end{proof}

The conclusion of the proof is now a simple consequence of the fact that if $T$ is a metric current such that $\tilde T$ is purely non-flat, then $\mu_{\tilde T}$ is singular with respect to $\Leb^d$.

\begin{teorema}\label{p:continuityflat}
    Let $T$ be a metric $1$-current in $\mathbb{R}^d$. Then $\tilde T$ is a flat chain.
\end{teorema}

\begin{proof}
Assume by contradiction that there exists a metric $1$-current $T$ such that $\tilde T$ is not a flat chain, or equivalently that there exists a Borel set $A$ such that $\tilde T\trace \mathbb{1}_A$ is a non-trivial purely non-flat current. We claim that  $T^s:=T\trace \mathbb{1}_A$ is a metric current such that $\widetilde{ T^s}=\tilde T\trace \mathbb{1}_A$. First, $T\trace \mathbb{1}_A$ is a metric current thanks to \cite[Theorem 3.5]{metriccurrents}. Let us prove the claimed identity. 
For every $(f,\pi)\in \mathscr{D}^1(\R^d)$ with $f,\pi\in C^1(\R^d)$ we have
$$\langle \widetilde{T^s},fd\pi\rangle= T^s(f,\pi)= T\trace \mathbb{1}_A(f,\pi)=T(f\mathbb{1}_A,\pi)=\langle\tilde{T},f\mathbb{1}_Ad\pi\rangle=\langle\tilde{T}\trace \mathbb{1}_A,fd\pi\rangle,$$
where the third identity above follows from \cite[Theorem 3.5]{metriccurrents}.
This proves the identity $\widetilde{T^s}=\tilde{T}\trace \mathbb{1}_A$ and hence $\widetilde{T^s}$ is purely non-flat.

We claim that the total variation of $\widetilde{T^s}$ is singular with respect to $\Leb^d$. Indeed, suppose that there exists a Borel set $B$ such that $\mu_{\tilde T^s}\trace B\ll\Leb^d$. In particular, it follows from the Lebesgue density theorem that the current $\tilde T^s\trace \mathbf{1}_B$ is a limit in mass of currents of the form
$$T_N:=\sum_{i=1}^N v_i \Leb^d\trace B_i,$$
where $B_i$ are balls and $v_i$ are constant on each $B_i$. It is easy to check that such $T_N$ is a normal current, which proves the claim.

By \cref{prop:singcurves} we can find arbitrarily small vectors $v\in\R^d$ such that $\mu_{\tilde T^s}$ and $(\tau_v)_\sharp\mu_{\tilde T^s}$ are mutually singular, so that, observing that $\tilde T^s-(\tau_v)_\sharp\tilde T^s$ is purely non-flat, we have $$\Flat(\tilde T^s-(\tau_v)_\sharp\tilde T^s)=\Mass(\tilde T^s-(\tau_v)_\sharp\tilde T^s)=2\Mass(\tilde T^s),$$ where the first equality follows from Proposition \ref{p:flatmass}. This however contradicts Proposition \ref{p:continuityflat}.
\end{proof}

\begin{osservazione}
    We remark that we used that $T$ is 1-dimensional only in \cref{l:reform}. Providing a suitable generalization of this lemma does not seem feasible, see for instance \cite{zbMATH05971405} and  \cite[Theorem 2.1]{zbMATH01156782}. This obstructions is deeply connected to the fact that Schauder estimates for the Laplacian fail for continuous data. However, when we apply \cref{l:reform} in \cref{propfudb}, we do not really need the equality $d\omega=\pi$, but only the equality between the action of $\tilde T$ on $\omega$ and the action of $T$ on $(1,\pi)$ when $\tilde T$ is purely non-flat. Therefore, in principle, the obstruction above does not exclude the possibility to adapt our strategy to prove the flat chain conjecture in full generality.  
\end{osservazione}

\section*{Conflict of interest and data availability}
On behalf of all authors, the corresponding author states that there is no conflict of interest.\\

The manuscript has no associated data.

\printbibliography

\end{document}